\documentclass[a4paper,11pt]{article}

\usepackage[utf8]{inputenc}
\usepackage[english,activeacute]{babel}
\usepackage{amsmath,amsthm,amsfonts,amssymb}
\usepackage{mathpartir}
\usepackage{ stmaryrd }
\usepackage{graphics}
\usepackage{enumerate}
\usepackage{mathrsfs}
\usepackage{rotating}

\input{xy}
\usepackage[all]{xy}

\newtheorem{thm}{Theorem}[section]
\newtheorem{cor}[thm]{Corollary}

\theoremstyle{plain}

\theoremstyle{remark}

\input{diagxy}
\xyoption{curve}

\newbox\anglebox 
\setbox\anglebox=\hbox{\xy \POS(75,0)\ar@{-} (0,0) \ar@{-} (75,75)\endxy}
 
\newbox\angleboxr 
\setbox\angleboxr=\hbox{\xy \POS(0,0)\ar@{-} (0,75) \ar@{-} (75,0)\endxy}
 
\newbox\sanglebox 
\setbox\sanglebox=\hbox{\xy \POS(50,0)\ar@{-} (0,0) \ar@{-} (50,50)\endxy}
 
\newbox\sangleboxr 
\setbox\sangleboxr=\hbox{\xy \POS(0,0)\ar@{-} (0,50) \ar@{-} (50,0)\endxy}
 
\newbox\sangleboxf 
\setbox\sangleboxf=\hbox{\xy \POS(50,50)\ar@{-} (50,0) \ar@{-} (0,50)\endxy}
 
\newbox\angleboxf 
\setbox\angleboxf=\hbox{\xy \POS(75,75)\ar@{-} (75,0) \ar@{-} (0,75)\endxy}
 
\newbox\sangleboxfr 
\setbox\sangleboxfr=\hbox{\xy \POS(0,50)\ar@{-} (50,50) \ar@{-} (0,0)\endxy}
 
\newbox\angleboxfr 
\setbox\angleboxfr=\hbox{\xy \POS(0,75)\ar@{-} (75,75) \ar@{-} (0,0)\endxy}

\newcommand{\Sets}{\ensuremath{\mathbf{Set}}}

\title{Preservation theorems for strong first-order logics}
\author{Christian Esp\'indola}

\begin{document}
\date{}
\maketitle

\begin{abstract}
We prove preservation theorems for $\mathcal{L}_{\omega_1, G}$, the countable fragment of Vaught's closed game logic. These are direct generalizations of the theorems of \L{}o\'s-Tarski (resp. Lyndon) on sentences of $\mathcal{L}_{\omega_1, \omega}$ preserved by substructures (resp. homomorphic images). The solution, in $ZFC$, only uses general features and can be extended to several variants of other strong first-order logic that do not satisfy the interpolation theorem; instead, the results on infinitary definability are used. This solves an open problem dating back to 1977. Another consequence of our approach is the equivalence of the Vop\v{e}nka principle and a general definability theorem on subsets preserved by homomorphisms.
\end{abstract}

\noindent $\mathbf{Keywords:}$ infinitary logics, preservation theorems, infinitary model theory, categorical logic

\section{Introduction}
  \L{}o\'s-Tarski preservation theorem for first-order logic states that if a sentence is preserved under substructures, it is equivalent to a universal sentence, i.e., one in which, in negation normal form, only contains universal quantifiers (see e.g., \cite{hodges}). This result is essentially a corollary of a more general result on sentences preserved under homomorphisms, which are equivalent to so called positive existential sentences (coherent sentences, in the terminology of \cite{johnstone}).
  Lyndon found another related preservation theorem for sentences preserved under homomorphic images (that is, surjective homomorphisms). Namely, any such sentence is equivalent to a positive sentence, i.e., a sentence for which, in negation normal form, no atomic formula occurs negated (\cite{lyndon}). Both \L{}o\'s-Tarski and Lyndon preservation theorems have been generalized to the infinitary logic $\mathcal{L}_{\omega_1, \omega}$ (see \cite{malitz}, \cite{le}, \cite{keisler}).
  Strong first-order logic are an extension of the logic of $\mathcal{L}_{\omega_1, \omega}$ described in \cite{burgess}, which share some of its properties, while others fail. In \cite{burgess} the failure of the interpolation theorem is shown, while the question on preservation theorems holding for, e.g., $\mathcal{L}_{\omega_1, G}$, are left open. From the last paragraph: 

``One large problem in the model theory of strong first-order languages remains open, which does not lend itself to abstract, descriptive-set-theoretic statement: Can we prove for, say, $\mathcal{L}_{\omega_1, G}$, that any sentence preserved under substructure (resp. homomorphic image) is equivalent to a universal (resp. positive) sentence?"

We give here a positive answer to both questions (universal and positive sentences) in the case of Vaught's game logic $\mathcal{L}_{\omega_1, G}$. The methods are however general enough to be carried out within $ZFC$ and to apply to a wider variety of the languages presented in \cite{burgess}. Instead of considering descriptive set-theoretic arguments, which encounter difficulties when analyzing preservation theorems, we will rely instead on a definability result here obtained with the aid of topos-theoretic techniques. More precisely, we work with $\lambda$-classifying toposes, introduced in \cite{espindolad}. These will also allow us to show the equivalence of Vop\v{e}nka principle with a general definability theorem on subsets preserved by homomorphisms.

  The structure of this work is as follows: we first present the topos-theoretic argument leading to our definability result, and later present its applications to the particular problem of preservation by homomorphisms. We assume that the reader is acquainted with the basic topos-theoretic machinery, particularly with section $D$ of \cite{johnstone} as well as familiarity with $\lambda$-coherent logic and $\lambda$-classifying toposes from \cite{espindola} and \cite{espindolad}. This is a continuation of the research project on infinitary first-order categorical logic started by the author in \cite{espindola}.

\section{The $\lambda$-classifying topos of a $\kappa$-theory}

In this section fix $\kappa<\lambda$ such that $\kappa^{<\kappa}=\kappa$ and $\lambda^{<\lambda}=\lambda$. Let $\mathbb{T}$ be a $\kappa$-coherent theory in $\mathcal{L}_{\kappa^+, \kappa}$, $\mathcal{C}_{\mathbb{T}}$ be its syntactic category and $Mod_{\lambda}(\mathbb{T})$ be the full subcategory of $\lambda$-presentable models. Assume that the category of models of $\mathbb{T}$ is $\lambda$-accessible (this is the case, e.g., if $\lambda=\kappa^+$ or, more generally, if $\kappa \vartriangleleft \lambda$). Let $\mathbb{T}'$ be the theory in $\mathcal{L}_{\lambda^+, \lambda}$ with the same axioms as those of $\mathbb{T}$. An important result we will prove here is the following:

\begin{thm}\label{lk}
The $\lambda$-classifying topos of $\mathbb{T}'$ is equivalent to the presheaf topos $\mathcal{S}et^{Mod_{\lambda}(\mathbb{T})}$. Moreover, the canonical embedding of the syntactic category $\mathcal{C}_{\mathbb{T}'} \hookrightarrow \mathcal{S}et^{Mod_{\lambda}(\mathbb{T})}$ is given by the evaluation functor, which on objects acts by sending $(\mathbf{x}, \phi)$ to the functor $\{M \mapsto [[\phi]]^M\}$.
\end{thm}

\begin{proof}
By hypothesis every model of $\mathbb{T}'$ is a $\lambda$-filtered colimit of models in $Mod_{\lambda}(\mathbb{T})$. Note first that the following diagram:

\begin{displaymath}
\xymatrix{
\mathcal{C}_{\mathbb{T}'} \ar@{^{(}->}[rr]^{ev} \ar@{->}[ddr]_{M \cong \varinjlim_i M_i} & & \Sets^{Mod_{\lambda}(\mathbb{T})} \ar@{->}[ddl]^{M' \cong \varinjlim_i ev_{M_i}}\\
 & & \\
 & \Sets & \\
}
\end{displaymath}
\noindent
commutes up to invertible $2$-cell. Here $ev$ and $ev_{M_i}$ are the evaluation functors, defined on objects as $ev((\mathbf{x}, \phi))=\{M \mapsto [[\phi]]^M\}$ and $ev_{M_i}(F)=F(M_i)$, respectively, while $\varinjlim M_i$ is the canonical $\lambda$-filtered colimit of $\lambda$-presentable models associated to the model $M$. Note also that since $\lambda$-filtered colimits commute with $\lambda$-small limits, $M'$ will preserve, in addition to all colimits, also $\lambda$-small limits.

Let now $\Sets[\mathbb{T}']_{\lambda}$ be the $\lambda$-classifying topos of $\mathbb{T}'$. We shall prove that this latter is equivalent to $\mathcal{S}et^{Mod_{\lambda}(\mathbb{T})}$ by verifying in this presheaf topos the universal property of $\Sets[\mathbb{T}']_{\lambda}$, i.e., that models of $\mathbb{T}'$ in a $\lambda$-topos $\mathcal{E}$ corresponds to $\lambda$-geometric morphisms from $\mathcal{E}$ to the presheaf topos. It is enough to prove this universal property in the particular case in which $\mathcal{E}=\Sets[\mathbb{T}']_{\lambda}$. 

  Given then the $\lambda$-classifying topos $\mathcal{E}$ of $\mathbb{T}'$, by the completeness theorem of \cite{espindolad} it will have enough $\lambda$-points. Hence, there is a conservative $\lambda$-geometric morphism with inverse image $E: \mathcal{E} \to \Sets^{I}$ such that composition with the evaluation at $i \in I$, $ev(i)E$ gives a $\lambda$-point of $\mathcal{E}$. Now each model of $\mathbb{T}'$ in $\mathcal{E}$, $N: \mathcal{C}_{\mathbb{T}'} \to \mathcal{E}$ gives rise to models in $\Sets$ by considering their images through each $ev(i)E$. These correspond to unique (up to isomorphism) $\lambda$-geometric morphisms with inverse image $\Sets^{Mod_{\lambda}(\mathbb{T})} \to \Sets$, which in turn induce a $\lambda$-geometric morphism with inverse image $G: \Sets^{Mod_{\lambda}(\mathbb{T})} \to \Sets^I$ and with the property that the composition $G \circ ev: \mathcal{C}_{\mathbb{T}'} \to \Sets^{Mod_{\lambda}(\mathbb{T})} \to \Sets^I$ is the same (up to isomorphism) as $EN: \mathcal{C}_{\mathbb{T}'} \to \Sets^I$. In other words, considering $\mathcal{E}$ as a subcategory of $\Sets^I$, the image of $G \circ ev$ belongs to $\mathcal{E}$.

\begin{displaymath}
\xymatrix{
\mathcal{C}_{\mathbb{T}'} \ar@{^{(}->}[rr]^{ev} \ar@{->}[ddr]_{N} & & \Sets^{Mod_{\lambda}(\mathbb{T})} \ar@{-->}[ddl] \ar@{->}[dddl]^{G} \\
 & & \\
 & \mathcal{E} \ar@{->}[d]_{E} & \\
 & \Sets^I \ar@{->}[d]_{ev(i)} & \\
  & \Sets & \\
}
\end{displaymath}
\noindent

  On the other hand, every object $F$ in $\Sets^{Mod_{\lambda}(\mathbb{T})}$ can be canonically expressed as a colimit of representables, $F \cong \varinjlim_i [M_i, -]$. In turn, each $M_i: \mathcal{C}_{\mathbb{T}} \to \Sets$ is a $\lambda$-small colimit of representables $M_i \cong \varinjlim_j [\phi_{ij}, -]$. It follows that:

$$F \cong \varinjlim_i[\varinjlim_j[\phi_{ij}, -]_{\mathcal{C}_{\mathbb{T}'}}, -]_{Mod_{\lambda}(\mathbb{T})} \cong \varinjlim_i \varprojlim_j [[\phi_{ij}, -]_{\mathcal{C}_{\mathbb{T}'}}, -]_{Mod_{\lambda}(\mathbb{T})} \cong \varinjlim_i \varprojlim_j ev(\phi_{ij})$$
\noindent
where the last isomorphism follows from Yoneda lemma. Now $G$ preserves $\lambda$-small limits and colimits, and so we will have:

$$G(F) \cong \varinjlim_i \varprojlim_j G \circ ev(\phi_{ij})$$
\noindent
and similarly on arrows. Therefore, $G$ is completely determined (up to isomorphism) by its value on the objects $ev(\phi_{ij})$. Since the value of $G$ on such objects belongs to $\mathcal{E}$, and $E$ preserves $\lambda$-small limits and colimits, it follows that $G$ itself factors through $\mathcal{E}$. Moreover, it is the unique (up to isomorphism) inverse image of a $\lambda$-geometric morphism corresponding to the given model in $\mathcal{E}$. This finishes the proof.
\end{proof}  

\section{Preservation theorems for $\mathcal{L}_{\omega_1, G}$}

The language $\mathcal{L}_{\omega_1, G}$ is the fragment of Vaught's closed game logic which extends $\mathcal{L}_{\omega_1, \omega}$ by allowing the following instance of infinitary quantification:

$$\forall x_0 \bigwedge_{b_0 \in I} \exists y_0 \bigvee_{c_0 \in I} \forall x_1 \bigwedge_{b_1 \in I} \exists y_1 \bigvee_{c_1 \in I} \qquad ... \qquad \bigwedge_{i<\omega} \phi^{b_0c_0b_1c_1...b_ic_i}_i(x_0, y_0, ..., x_i, y_i) \qquad (1)$$
\noindent
There is a game semantics associated to the sentence (1) as follows: the first player chooses an element $x_0$ and a conjunct $b_0$, then the second player chooses an element $y_0$ and a disjunct $c_0$, and the game continues with $\omega$ many moves, after which the second player wins if with the choices made during the game it is the case that each $\phi^{b_0c_0b_1c_1...b_ic_i}_i(x_0, y_0, ..., x_i, y_i)$ is satisfied in the structure $M$, for every $i<\omega$. Since the formula in the matrix corresponds to a closed subset of $|M|^{\omega} \times I^{\omega}$, by determinacy for closed games it follows that the game is determined, and hence the formula is said to be true in $M$ if the second player has a winning strategy, while it is said to be false if the first player has a winning strategy, i.e., if:

$$\exists x_0 \bigvee_{b_0 \in I} \forall y_0 \bigwedge_{c_0 \in I} \exists x_1 \bigvee_{b_1 \in I} \forall y_1 \bigwedge_{c_1 \in I} \qquad ... \qquad \bigvee_{i<\omega} \neg \phi^{b_0c_0b_1c_1...b_ic_i}_i(x_0, y_0, ..., x_i, y_i) \qquad (2)$$
\noindent
holds. The formula (1) generates a fragment within Vaught's closed game logic closed under finitary connectives and operations and containing all subformulas of (1). We have now:

\begin{thm}\label{universal}
A sentence of $\mathcal{L}_{\omega_1, G}$ which is preserved under substructures is equivalent to a universal sentence of $\mathcal{L}_{\omega_1, G}$.
\end{thm}

\begin{proof}
We will prove the dual statement, namely, that sentences preserved upwards along embeddings are equivalent to existential sentences. Let $\phi$ be a sentence of $\mathcal{L}_{\omega_1, G}$ which is preserved by embeddings, and assume without loss of generality that the language is relational. For each relation $R$, including equality, in the signature (which we can assume countable by passing to a fragment generated by $\phi$), add a new relation $R^*$ together with the theory $\mathbb{T}$ consisting of the following axioms:

$$R(\mathbf{x}) \wedge R^*(\mathbf{x}) \vdash_{\mathbf{x}} \bot$$
$$\top \vdash_{\mathbf{x}} R(\mathbf{x}) \vee R^*(\mathbf{x})$$

The homomorphisms in the new language will then correspond to embeddings and by hypothesis $\phi$ is preserved. Assume first that the continuum hypothesis holds. Then, by Theorem \ref{lk} the $\omega_1$-classifying topos of $\mathbb{T}$ is the topos $\mathcal{S}et^{Mod_{\omega_1}(\mathbb{T})}$ of presheaves over the subcategory of (at most) countable models and embeddings. The interpretation of $\phi$ in each such model $M$, say, $[[\phi]]^M$, defines therefore a subobject of $S \rightarrowtail [[(\{\}, \top)]]$ in the topos. Since the embedding $\mathcal{C}_{\mathbb{T}'} \hookrightarrow \mathcal{S}et^{Mod_{\lambda}(\mathbb{T})}$ can be identified with Yoneda embedding $\mathcal{C}_{\mathbb{T}'} \hookrightarrow \mathcal{S}h(\mathcal{C}_{\mathbb{T}'}, \tau)$ with the $\omega_1$-coherent topology, $S$ corresponds to a union of representable subobjects, and so it is the interpretation of some $\omega_1$-coherent formula of the form $\bigvee_{j<\omega_1} \exists_{i<\omega} x_i \bigwedge_{i<\omega} \psi_i^j$, where the $\psi_i^j$ are atomic formulas with free variables amongst $x_0, ..., x_{i-1}$. But this formula can be rewritten as the following formula $\psi$ in the original signature:

$$\bigvee_{j_0<\omega} \exists x_0 \bigvee_{j_1<\omega} \exists x_1 \qquad ... \qquad \bigwedge_{i<\omega} \phi_i^{j_0...j_i}$$ 
\noindent
where we identify each $j<\omega_1$ with the subset $\{j_0, j_1, ...\} \subseteq \omega$ and each $\phi_i^{j_0...j_i}$ is obtained by simply replacing in $\bigwedge_{k<i}\bigvee_{j|_{i+1}=\{j_0, ..., j_i\}}\psi_k^j$ each relation symbol $R^*$ with $\neg R$ and reducing the size of the disjunctions to $\omega$ (this is possible since there are at most countable many $\psi_i^j$). The resulting formula $\psi$ is now in $\mathcal{L}_{\omega_1, G}$ (is in fact a Vaught sentence), is clearly existential, and its interpretation coincides with that of $\phi$ in all countable models. We claim that it is actually equivalent to $\phi$. Indeed, the formula $\phi \leftrightarrow \psi$ admits an approximation by formulas in $\mathcal{L}_{\infty, \omega}$ (see \cite{burgess}), i.e., there are formulas $\mathcal{A}(\phi \leftrightarrow \psi, \alpha)$ in $\mathcal{L}_{\infty, \omega}$ such that $\phi \leftrightarrow \psi$ is equivalent to the formal conjunction $\bigwedge_{\alpha \in Ord}\mathcal{A}(\phi \leftrightarrow \psi, \alpha)$. Hence, if $\phi \leftrightarrow \psi$ was not valid, we would have:

$$\exists M \exists \alpha M \nvDash \mathcal{A}(\phi \leftrightarrow \psi, \alpha).$$
\noindent
This is a $\Sigma_1$ sentence, so that since $\phi \leftrightarrow \psi$ is in $H(\omega_1)$, the set hereditarily of cardinality at most countable, by Shoenfield-Levy's theorem we can assume that $M$ and $\alpha$ are in $H(\omega_1)$, which would contradict our previous result.

  Suppose now that the value of the continuum is arbitrary. Consider the forcing extension $V[G]$ in which we collapse $2^{\omega}$ to $\omega_1$. Since this forcing is $<\omega_1$-distributive, formulas of $\mathcal{L}_{\omega_1, G}$ and their countable models and embeddings remain unchanged (we assume they are properly coded). By what we have just proved, $\phi$ is equivalent in $V[G]$ to an existential formula $\psi$, and since the validity of $\phi \leftrightarrow \psi$ is a $\Pi_1$ formula, it will be true in the ground model, which finishes our proof.
\end{proof}

\begin{thm}\label{positive}
A sentence of $\mathcal{L}_{\omega_1, G}$ which is preserved under homomorphic images is equivalent to a positive sentence of $\mathcal{L}_{\omega_1, G}$.
\end{thm}

\begin{proof}
We proceed as before for this case as well; in particular, it is enough to prove that, assuming the continuum hypothesis, every sentence $\phi$ of $\mathcal{L}_{\omega_1, G}$ which is preserved under homomorphic images between countable models is equivalent, on countable models, to a positive sentence $\psi$ of $\mathcal{L}_{\omega_1, G}$. Then we can prove the general case as we did in the proof of Theorem \ref{universal}.
  Extend the signature by adding countable many constant symbols $(c_i)_{i<\omega}$ and a relation symbol $S$,  and consider the theory $\mathbb{T}$ axiomatized by the sequent:

$$\top \vdash_{y} \bigvee_{i<\omega}y=c_i$$

This is an $\omega$-coherent theory and its homomorphisms are evidently surjective, so that $\phi$ is preserved. In an entirely similar way as with the proof of Theorem \ref{universal}, we deduce then that there is a $\omega_1$-coherent formula $\psi_0$ of the form $\bigvee_{j<\omega_1} \exists_{i<\omega} x_i \bigwedge_{i<\omega} \psi_i^j$, where the $\psi_i^j$ are atomic formulas, and which is equivalent to $\phi$ in all countable models of $\mathbb{T}$; that is:

$$\forall y \bigvee_{i<\omega} y=c_i \vDash \phi  \leftrightarrow \psi_0(\mathbf{c})$$
\noindent
where $\mathbf{c}=c_0c_1 ...$, or:

$$\models \forall \mathbf{x} \left[\forall y \bigvee_{i<\omega} y=x_i \to (\phi  \leftrightarrow \psi_0(\mathbf{x}))\right] \qquad (3)$$
\noindent
where $\mathbf{x}=x_0x_1 ...$. Now (3) readily implies that $\phi \models \forall \mathbf{x} (\forall y \bigvee_{i<\omega} y=x_i \to \psi_0(\mathbf{x}))$. In countable models, this latter sentence $\forall \mathbf{x} (\forall y \bigvee_{i<\omega} y=x_i \to \psi_0(\mathbf{x}))$ entails:

$$\exists \mathbf{x} \left(\forall y \bigvee_{i<\omega} y=x_i \wedge \psi_0(\mathbf{x})\right) \qquad (4)$$
\noindent
since that sentence and the negation of (4) implies $\forall \mathbf{x} \neg (\forall y \bigvee_{i<\omega} y=x_i)$, which is only true in uncountable models. On the other hand, using (3) we see that (4) clearly implies $\phi$ in all models. Thus, we have that in all countable models, $\phi$ is equivalent to the sentence (4). This sentence is clearly positive, but it does not belong to $\mathcal{L}_{\omega_1, G}$. To find an appropriate sentence in $\mathcal{L}_{\omega_1, G}$, note that (4) is equivalent to the following second-order sentence:

$$\exists_{i<\omega}R_i \left[\forall x \bigvee_{i<\omega}R_i(x) \wedge \bigwedge_{i<\omega}(R_i(x) \wedge R_i(y) \to x=y) \wedge \exists \mathbf{x}\left(\bigwedge_{i<\omega}R_i(x_i) \wedge \psi_0(\mathbf{x})\right)\right] \qquad (5)$$
\noindent
where the $(R_i)_{i<\omega}$ are unary relations whose sole purpose is to code the constants $(c_i)_{i<\omega}$, i.e., they are such that the intended interpretation of $R_i(x)$ is $c_i$. Now (5) expresses $\phi$ as a projective class over $\mathcal{L}_{\omega_1, G}$, since the conjunct $\exists \mathbf{x}(\bigwedge_{i<\omega}R_i(x_i) \wedge \psi_0(\mathbf{x}))$ can clearly be rewritten, as we did in the proof of Theorem \ref{universal}, as a sentence of $\mathcal{L}_{\omega_1, G}$. Note also that in (5) (or rather, in its rewritten form in $\mathcal{L}_{\omega_1, G}$) every atomic formula not involving the $R_i$ which are quantified, appears positively in negation normal form. It follows by results of Vaught from \cite{vaught} that the matrix of (5) (i.e., the formula after the existentially quantified $R_i$) is equivalent in turn to a second-order assertion of the form $\exists_{j<\omega} S_j \theta$, where $\theta$ is in $\mathcal{L}_{\omega_1, \omega}$ and has the property that every atomic formula appears positively in negation normal form. Hence, (5) actually expresses $\phi$ as a projective class over $\mathcal{L}_{\omega_1, \omega}$. It follows also from \cite{vaught} that this resulting second-order assertion is equivalent over countable models to a Vaught sentence $\psi$ in which every atomic formula appears positively in negation normal form\footnote{Indeed, see the comments in \cite{vaught} starting from the last paragraph of page 18. Alternatively, by a result of Makkai, a Vaught sentence preserved by homomorphic images is equivalent to a positive Vaught sentence.}, i.e., a positive sentence of $\mathcal{L}_{\omega_1, G}$. This finishes the proof.
\end{proof}

\begin{cor}
A sentence of $\mathcal{L}_{\omega_1, \omega}$ preserved under substructure (resp. homomorphic image) is equivalent to a universal (resp. positive) sentence of $\mathcal{L}_{\omega_1, \omega}$.
\end{cor}

\begin{proof}
Since $\mathcal{L}_{\omega_1, \omega} \subset \mathcal{L}_{\omega_1, G}$, such a sentence $\phi$ is equivalent to a universal (resp. positive) Vaught sentence $\Phi$. By Vaught's covering theorem (see, e.g., \cite{vaught}), since $\Phi \to \phi$, there is a countable ordinal $\alpha$ such that $\mathcal{A}(\Phi, \alpha) \to \phi$. Thus, $\phi$ is equivalent to the sentence $\mathcal{A}(\Phi, \alpha)$ which is in $\mathcal{L}_{\omega_1, \omega}$ and is universal (resp. positive).
\end{proof}

\section{A definability property equivalent to Vop\v{e}nka principle}

As a final application of Theorem \ref{lk}, we now prove the following:

\begin{thm}
Let $\Sigma$ be a signature and consider the category of $\Sigma$-structures and homomorphisms. Suppose that for each structure $M$ there is a distinguished subset $S^M$ which is preserved by all homomophisms. Then the following are equivalent:

\begin{enumerate}
\item Vop\v{e}nka principle
\item The subsets $S^M$ are definable by an infinitary coherent formula. That is, there is a formula $\phi$ of the form $\bigvee_{j<\lambda} \exists_{i<\kappa} x_i \bigwedge_{i<\kappa} \psi_i^j$, where the $\psi_i^j$ are atomic formulas, such that $[[\phi]]^M=S^M$ for all $\Sigma$-structures $M$.
\end{enumerate}
\end{thm}

\begin{proof}
$(2 \implies 1)$ This part is essentially contained in \cite{ar}. If  Vop\v{e}nka principle does not hold, there is a large rigid class of structures $\mathcal{C}$ (see \cite{ar}). Define now:

$$S^M=\bigcup_{N \in \mathcal{C}} \bigcup_{f: N \to M} im(f)$$
\noindent
If $S$ is the subfunctor of the identity defined by the $S^M$, then $S$ is not accessible (see Remark in page 268 of \cite{ar}). Hence, it cannot be definable, as every definable subfunctor (by an infinitary coherent formula) is clearly accessible.

$(1 \implies 2)$ Assuming that Vop\v{e}nka principle holds, the subfunctor $S$ is accessible (since then a subfunctor of an accessible functor must be acccessible). Choose an inaccessible $\lambda$ such that $S$ is $\lambda$-accessible. By Theorem \ref{lk}, the $\lambda$-classifying topos of the empty theory over $\Sigma$ is the presheaf topos $\mathcal{S}et^{Mod_{\lambda}(\mathbb{T})}$. Analogously as to what we did in the previous section, it follows that the subfunctor coincides in all models of size less than $\lambda$ with the interpretation of a $\lambda$-coherent formula. Since this latter is computed in a model $M$ as the $\lambda$-filtered colimit of its value on $\lambda$-presentable models, it follows that the equivalence holds in all $\Sigma$-structures. This concludes the proof.
\end{proof}

\section{Acknowledgements}

This research has been supported through the grant 19-00902S from the Grant Agency of the Czech Republic.

\bibliographystyle{amsalpha}

\renewcommand{\bibname}{References} 

\bibliography{references}



\end{document}